\documentclass{amsart}
\usepackage{latexsym,rotate,eucal,cite}
\usepackage{amsmath,amsthm,amssymb,amsxtra}
\usepackage[pagewise]{lineno}
\usepackage{cite}
\newtheorem{theorem}{Theorem}[section]
\newtheorem{lemma}[theorem]{Lemma}
\newtheorem{proposition}[theorem]{Proposition}
\newtheorem{corollary}[theorem]{Corollary}

\theoremstyle{definition}

\theoremstyle{remark}
\newtheorem{remark}[theorem]{Remark}

\numberwithin{equation}{section}

\begin{document}

\title{Biharmonic hypersurfaces with constant scalar curvature in space forms}

\author{Yu Fu}
\address{School of Mathematics, Dongbei University of Finance and Economics,
Dalian 116025, P. R. China}
\email{yufudufe@gmail.com}

\author{Min-Chun Hong}
\address{Department of Mathematics, The University of Queensland, Brisbane,
QLD 4072, Australia}
\email{hong@maths.uq.edu.au}





\begin{abstract}
Let $M^n$ be a biharmonic hypersurface with constant scalar
curvature in a space form $\mathbb M^{n+1}(c)$. We show that $M^n$
has constant mean curvature if $c>0$ and $M^n$ is minimal if
$c\leq0$, provided that the number of distinct principal curvatures
is no more than 6. This partially confirms Chen's conjecture and Generalized
Chen's conjecture. As  a consequence, we prove that there exist
no proper biharmonic hypersurfaces with constant scalar
curvature in Euclidean space $\mathbb E^{n+1}$ or hyperbolic space $\mathbb H^{n+1}$ for $n<7$.
\end{abstract}

\maketitle
\markboth{Fu and Hong}
{Biharmonic hypersurfaces with constant scalar curvature}
\section{Introduction}
In 1983,  Eells and  Lemaire \cite{Eells} introduced the concept of
{\em biharmonic maps} in order to generalize classical theory of
harmonic maps. A biharmonic map $\phi$  between an $n$-dimensional
Riemannian manifold  $(M^n,g)$ and an $m$-dimensional Riemannian
manifold $(N^m,h)$ is  a critical point of the bienergy functional
\begin{eqnarray*}
E_2(\phi)=\frac{1}{2}\int_M|\tau(\phi)|^2dv_g,
\end{eqnarray*}
where $\tau(\phi)= {\rm trace \nabla d\phi}$ is the tension field of
$\phi$ that vanishes for a harmonic map.
More clearly, the Euler-Lagrange equation associated to the
bienergy is given by
\begin{eqnarray*}
\tau_2(\phi)=-\Delta\tau(\phi)-{\rm trace}\,
R^{N}(d\phi,\tau(\phi))d\phi=0,
\end{eqnarray*}
where $R^{N}$ is the curvature tensor of $N^m$ (e.g.\cite{jiang1987}). We call  $\phi$  to be  a biharmonic map if its
bitension field $\tau_2(\phi)$ vanishes.

Biharmonic maps between Riemannian manifolds have been extensively
studied by some geometers. In particular, many authors investigated
a special class of biharmonic maps named {\em biharmonic
immersions}. An immersion $\phi:(M^n,g)\longrightarrow(N^m,h)$ is
biharmonic  if and only if its mean curvature vector field
$\overrightarrow{H}$ fulfills the fourth-order semi-linear elliptic
equations  (e.g. \cite{CMO2001})
\begin{eqnarray}
\Delta\overrightarrow{H}+{\rm trace}\,
R^{N}(d\phi,\overrightarrow{H})d\phi=0.
\end{eqnarray}
It is well-known that any minimal immersion (satisfying $\overrightarrow{H}=0$) is harmonic. The
non-harmonic biharmonic immersions are called proper biharmonic.

We should mention that {\em biharmonic submanifolds} in  a Euclidean
space $\mathbb E^m$ were independently defined by B. Y. Chen  in the middle of 1980s (see \cite{Chen1991}) with
the geometric condition $\Delta \overrightarrow{H}=0$ , or equivalently $\Delta^2 \phi=0$.
Interestingly, both  biharmonic
submanifolds  and biharmonic immersions in Euclidean spaces coincide  with each other.

In recent years,  the classification
problem of biharmonic submanifolds  has attracted a great attention in geometry. In particular,
there is a longstanding conjecture on biharmonic submanifolds due to B. Y. Chen \cite{Chen1991} in 1991:

{\bf Chen's conjecture}: {\em Every biharmonic submanifold in
Euclidean space $\mathbb E^m$ is minimal}.

Until now, Chen's conjecture remains open, even for hypersurfaces.
Only partial answers to Chen's conjecture have been obtained for
more than three decades, e.g. \cite{akutagawa2013},
\cite{alias2013}, \cite{chenbook2015},\cite{Ou20121}.
 In the case of hypersurfaces,
Chen's conjecture is true for the following special cases:
\begin{itemize}
\item surfaces in $\mathbb E^3$ \cite{Chen1991}, \cite{jiang1987};
\item hypersurfaces with at most two distinct principal curvatures in $\mathbb E^m$ \cite{Dimitri1992};
\item hypersurfaces in $\mathbb E^4$ \cite{HasanisVlachos1995} (see also\cite{defever1998});

\item $\delta(2)$-ideal and
$\delta(3)$-ideal hypersurfaces in $\mathbb E^m$\cite{chenMunteanu2013};

\item weakly convex hypersurfaces in $\mathbb E^m$\cite{luoyong2014};

\item hypersurfaces with at most three distinct principal curvatures in $\mathbb E^m$
\cite{fu1};

\item generic hypersufaces with irreducible principal curvature vector fields in $\mathbb E^m$ \cite{KoisoUrakawa2014};

\item invariant hypersurfaces of cohomogeneity one in $\mathbb E^m$ \cite{montaldo}.
\end{itemize}

In 2001, Caddeo, Montaldo and Oniciuc \cite{CMO2001} proposed the following generalized Chen's
conjecture:

{\bf Generalized Chen's conjecture}: {\em Every biharmonic
submanifold in a Riemannian manifold with non-positive sectional
curvature is minimal}.

Recently, Ou and  Tang in \cite{Ou2012} constructed a family of
counter-examples that  the generalized Chen's conjecture is false
when the ambient space has non-constant negative sectional
curvature. However,  the generalized Chen's conjecture remains open
when the ambient spaces have constant sectional curvature. For more
recent developments of the generalized Chen's conjecture, we refer
to \cite{chenbanyen2014}, \cite{chenbook2015}, \cite{montaldo2006},
 \cite{Oniciuc2012}, \cite{nakauchiurakawa}, \cite{Ou2016}.

We should point out that the classification of proper biharmonic
submanifolds in Euclidean spheres is rather rich and interesting.
The first example of proper biharmonic hypersurfaces is a
generalized Clifford torus $S^p(\frac1{\sqrt2})\times
S^q(\frac1{\sqrt2})\hookrightarrow \mathbb S^{n+1}$ with $p\neq q$
and $p+q=n$ given by Jiang \cite{jiang1986}. The complete
classifications of biharmonic hypersurfaces in $\mathbb S^3$ and
$\mathbb S^4$ were obtained in \cite{CMO2001}, \cite{BMO20102}.
Moreover, biharmonic hypersurfaces with at most three distinct
principal curvatures in $\mathbb S^n$ were classified in
\cite{BMO20102}, \cite{fu2}. For more details, we refer the readers
to Balmus, Caddeo, Montaldo, Oniciuc et al.'s work \cite{BMO2013},
\cite{loubeau}, \cite{Oniciuc2002}, \cite{Oniciuc2012},
\cite{ichinoura2010}.

In general, the classification problem of proper biharmonic
hypersurfaces in space forms becomes more complicated when the
number of distinct principal curvatures is four or more.

In view of the above aspects, it is reasonable to study biharmonic
submanifolds with some geometric conditions. In geometry,
hypersurfaces with constant scalar curvature have been intensively
studied by many geometers for the rigidity problem  and
classification problem, for instance, see the well-known paper of
Cheng-Yau \cite{Chengyau1977}. Some estimate for scalar curvature of
compact proper biharmonic hypersurfaces with constant scalar
curvature in spheres was obtained in \cite{BMO2008}. Recently, it
was proved in \cite{fu3} that a biharmonic hypersurface with
constant scalar curvature in the 5-dimensional space forms $\mathbb
M^{5}(c)$ necessarily has constant mean curvature.

Motivated by above results, in this paper we consider biharmonic hypersurfaces $M^n$
with constant scalar curvatures in a
space form $\mathbb M^{n}(c)$.  More precisely, we obtain:
\begin{theorem}
Let $M^n$ be an orientable biharmonic hypersurface with at most six
distinct principal curvatures in $\mathbb M^{n+1}(c)$. If the scalar curvature $R$ is constant, then $M^n$ has
constant mean curvature.
\end{theorem}
In general,
it is difficult to deal  with the biharmonic immersion equation (1.1)  due to its high nonlinearity. In order to prove
Theorem 1.1, we use some new ideas to overcome the difficulty of
treating the equation of a biharmonic hypersurface. More precisely,
we transfer the problem into a system of algebraic equations (see
Lemma 3.3), so  we can determine the behavior of the principal
curvature functions by investigating the solution of the system of
algebraic equations (see Lemma 3.4). Then, we are able to
prove that a biharmonic hypersurface  with constant scalar curvatures
in a space form $\mathbb M^{n}(c)$ must have constant mean
curvature, provided that the number of distinct principal curvature
is no more than six. We would like to point out that our approach in this paper is
different from those in \cite{fu2}, \cite{fu3},
\cite{defever1998}, \cite{BMO20102}.

\begin{remark}
Balmus-Montaldo-Oniciuc in \cite{BMO2008} conjectured that
the proper biharmonic hypersurfaces in $\mathbb S^{n+1}$ must
have constant mean curvature. Theorem 1.1 with $c=1$ gives a
partial answer to this conjecture.

We should  point out that the complete classification of proper
biharmonic hypersurfaces with constant mean curvature in a sphere is
still open for the case that the number of distinct principal
curvatures is more than three (cf. \cite{Oniciuc2012}).
\end{remark}

Moreover, combining these
results with the biharmonic equations in Section 2, we have:
\begin{corollary}
Any biharmonic hypersurface with constant scalar
curvature and with at most six distinct principal curvatures in Euclidean space $\mathbb E^{n+1}$ or
hyperbolic space $\mathbb H^{n+1}$ is minimal.
\end{corollary}

Thus, this result gives a partial answer to Chen's conjecture and
the generalized Chen's conjecture.

Furthermore, as a direct consequence, we  get the following characterization
result:
\begin{corollary}
Any biharmonic hypersurface with constant scalar
curvature in Euclidean space $\mathbb E^{n+1}$ or
hyperbolic space $\mathbb H^{n+1}$ for $n<7$ has to be minimal.
\end{corollary}
\begin{remark}
We could replace or weaken the condition {\em constant scalar
curvature} in Theorem 1.1 by {\em constant length of the second
fundamental form} or {\em linear Weingarten type}, i.e. the scalar
curvature $R$ satisfying $R=aH+b$ for some constants $a$ and $b$. In
fact, the discussion is extremely similar to the proof of Theorem
1.1 and the same conclusion holds true as well.
\end{remark}

The paper is organized as follows. In Section 2, we recall some
necessary background for theory of hypersurfaces and equivalent
conditions for biharmonic hypersurfaces. In Section 3, we prove some
useful lemmas (Lemma 3.1-Lemma 3.6), which are crucial to prove the
main theorem. Finally, in Section 4, we give a proof of Theorem 1.1.

\section{Preliminaries}

In this section, we recall some basic material for the theory of
hypersurfaces immersed in a Riemannian space form.

Let $\phi: M^n\rightarrow\mathbb{M}^{n+1}(c)$ be an isometric
immersion of a hypersurface $M^n$ into a space form
$\mathbb{M}^{n+1}(c)$ with constant sectional curvature $c$. Denote
the Levi-Civita connections of $M^n$ and $\mathbb{M}^{n+1}(c)$ by
$\nabla$ and $\tilde\nabla$, respectively. Let $X$ and $Y$ denote
the vector fields tangent to $M^n$ and let $\xi$ be a unit normal
vector field. Then the Gauss and Weingarten formulas (cf.
\cite{chenbook2015}) are given respectively by
\begin{align}
\tilde\nabla_XY&=\nabla_XY+h(X,Y),\label{l23}\\
\tilde\nabla_X\xi&=-AX,\label{l16}
\end{align}
where $h$ is the second fundamental form  and $A$ is the Weingarten
operator. Note that the second fundamental form $h$ and
the Weingarten operator $A$ are related by
\begin{eqnarray}\label{l3}
\langle h(X,Y),\xi\rangle=\langle AX,Y\rangle.
\end{eqnarray}
The mean curvature vector field $\overrightarrow{H}$ is defined by
\begin{eqnarray}\label{md}
\overrightarrow{H}=\frac{1}{n}{\rm trace}~h.
\end{eqnarray}
Moreover, the Gauss and Codazzi equations are given  respectively
by
\begin{eqnarray*}
R(X,Y)Z=c\big(\langle Y, Z\rangle X-\langle X,Z\rangle Y\big)+\langle
AY,Z\rangle AX-\langle AX,Z\rangle AY,
\end{eqnarray*}
\begin{eqnarray*}
(\nabla_{X} A)Y=(\nabla_{Y} A)X,
\end{eqnarray*}
where $R$ is the curvature tensor of $M^n$ and
$(\nabla_XA)Y$ is given by
\begin{eqnarray}\label{l7}
(\nabla_XA)Y=\nabla_X(AY)-A(\nabla_XY)
\end{eqnarray}
for all $X, Y, Z$ tangent to $M^n$.

Assume that $\overrightarrow{H}=H\xi$ and $H$ denotes the mean
curvature.

By identifying the tangent and the normal parts of the biharmonic
condition (1.1) for hypersurfaces in a space form $\mathbb
M^{n+1}(c)$, the following characterization result for $M^n$ to be
biharmonic was obtained (see also \cite{CMO2002}, \cite{BMO20102}).
\begin{proposition}
The immersion $\phi: M^n\rightarrow\mathbb{M}^{n+1}(c)$ of a
hypersurface $M^n$ in an $n+1$-dimensional space form $\mathbb
M^{n+1}(c)$ is biharmonic if and only if
\begin{equation}
\begin{cases}
\Delta H+H {\rm trace}\, A^2 =ncH,\\
2A\,{\rm grad}H+nH{\rm grad}H=0.
\end{cases}
\end{equation}
\end{proposition}
The Laplacian operator $\Delta$ on $M^{n}$ acting on a smooth
function $f$ is given by
\begin{eqnarray}
\Delta f=-\mathrm {div}(\nabla f)=-\sum_{i=1}^{n}<\nabla_{e_{i}}(\nabla f),
e_{i}>=-\sum_{i=1}^{n}(e_{i}e_{i}-\nabla_{e_{i}}e_{i})f.
\end{eqnarray}
The following result was obtained in \cite{fu2}.
\begin{theorem}
Let $M^n$ be an orientable proper biharmonic hypersurface with at
most three distinct principal curvatures in $\mathbb M^{n+1}(c)$.
Then $M^n$ has constant mean curvature.
\end{theorem}

\section{Some lemmas}
We now consider an orientable biharmonic hypersurface $M^n$ $(n>3)$
in a space form $\mathbb M^{n+1}(c)$.

In general, the set $M_A$ of all points of $M^n$, at which the
number of distinct eigenvalues of the Weingarten operator $A$ (i.e.
the principal curvatures) is locally constant, is open and dense in
$M^n$. Since $M^n$ with at most three distinct principal curvatures
everywhere in a space form $\mathbb M^{n+1}(c)$ is CMC, i.e. the mean
curvature is constant (Theorem 2.2), one can work only on the
connected component of $M_A$ consisting by points where the number
of principal curvatures is more than three (by passing to the limit,
$H$ will be constant on the whole $M^n$). On that connected
component, the principal curvature functions of $A$ are always
smooth.

Suppose that, on the component, the mean curvature $H$ is not
constant. Thus, there is a point $p$ where ${\rm grad}\,H (p)\neq0$.
In the following, we will work on an neighborhood of $p$ where ${\rm
grad}\,H (p)\neq0$ at any point of $M^n$.

The second equation of (2.6) shows that ${\rm grad}\,H$ is an
eigenvector of the Weingarten operator $A$ with the corresponding
principal curvature $-nH/2$. We may choose $e_1$ such that $e_1$ is
parallel to ${\rm grad}\,H$, and with respect to some suitable
orthonormal frame $\{e_1,\ldots, e_n\}$, the Weingarten operator $A$
of $M$ takes the following form
\begin{eqnarray}
A=\mathrm{diag}(\lambda_1,\lambda_2, \ldots, \lambda_n),
\end{eqnarray}
where $\lambda_i$ are the principal curvatures and
$\lambda_1=-nH/2$. Therefore, it follows from \eqref{md} that
$\sum_{i=1}^n\lambda_i=nH$, and hence
\begin{eqnarray}
\sum_{i=2}^n\lambda_i=-3\lambda_1.
\end{eqnarray}
Denote by $R$ the scalar curvature and by $B$ the squared length of
the second fundamental form $h$ of $M$. It follows from (3.1) that
$B$ is given by
\begin{eqnarray}
B={\rm trace}\, A^2
=\sum_{i=1}^n\lambda^2_i=\sum_{i=2}^n\lambda^2_i+\lambda^2_1.
\end{eqnarray}
From the Gauss equation, the scalar curvature $R$ is given by
\begin{eqnarray}
R=n(n-1)c+n^2H^2-B=n(n-1)c+3\lambda_1^2-\sum_{i=2}^n\lambda^2_i.
\end{eqnarray}
Hence
\begin{eqnarray}
\sum_{i=2}^n\lambda^2_i=n(n-1)c-R+3\lambda^2_1.
\end{eqnarray}
Since ${\rm grad}\,H=\sum_{i=1}^ne_i(H)e_i$ and $e_1$ is parallel to
${\rm grad}\,H$, it follows that
\begin{eqnarray*}
e_1(H)\neq0,\quad e_i(H)=0, \quad 2\leq i\leq n,
\end{eqnarray*}
and hence
\begin{eqnarray}
e_1(\lambda_1)\neq0,\quad e_i(\lambda_1)=0, \quad 2\leq i\leq n.
\end{eqnarray}
 Put $ \nabla_{e_i}e_j=\sum_{k=1}^n\omega_{ij}^ke_k$
$(1\leq i,j\leq n)$. A direct computation concerning the
compatibility conditions $\nabla_{e_k}\langle e_i,e_i\rangle=0$ and
$\nabla_{e_k}\langle e_i,e_j\rangle=0$ $(i\neq j)$ yields
respectively that
\begin{eqnarray}
\omega_{ki}^i=0,\quad \omega_{ki}^j+\omega_{kj}^i=0,\quad i\neq j.
\end{eqnarray}
The Codazzi equation could yield to
\begin{eqnarray}
e_i(\lambda_j)=(\lambda_i-\lambda_j)\omega_{ji}^j,\\
(\lambda_i-\lambda_j)\omega_{ki}^j=(\lambda_k-\lambda_j)\omega_{ik}^j
\end{eqnarray}
for distinct $i, j, k$.

Moreover, from (3.6) we have
\begin{eqnarray*}
[e_i,e_j](\lambda_1)=0,
\end{eqnarray*}
which yields directly
\begin{eqnarray}
\omega_{ij}^1=\omega_{ji}^1, \quad 2\leq i, j\leq n ~\,{\rm and}\,
~i\neq j.
\end{eqnarray}
\begin{lemma}
Let $M^n$ be an orientable biharmonic hypersurface with non-constant
mean curvature in $\mathbb M^{n+1}(c)$. Then the multiplicity of the
principal curvature $\lambda_1$ $(=-nH/2)$ is one, i.e.
$\lambda_j\neq\lambda_1$ for $2\leq j\leq n$.
\end{lemma}
\begin{proof}
If $\lambda_j=\lambda_1$ for $j\neq1$, by putting $i=1$ in (3.8) we
get
\begin{eqnarray*}
0=(\lambda_1-\lambda_j)\omega_{j1}^j=e_1(\lambda_j)=e_1(\lambda_1),
\end{eqnarray*}
which contradicts to (3.6).
\end{proof}
\begin{lemma}
The smooth real-valued functions $\lambda_i$ and $\omega_{ii}^1$
$(2\leq i\leq n)$ satisfy the following differential equations
\begin{align}
e_1e_1(\lambda_1)&=e_1(\lambda_1)\Big(\sum_{i=2}^n\omega_{ii}^1\Big)+\lambda_1\big(n(n-2)c-R+4\lambda_1^2\big),\\
e_1(\lambda_i)&=\lambda_i\omega_{ii}^1-\lambda_1\omega_{ii}^1,\\
e_1(\omega_{ii}^1)&=(\omega_{ii}^1)^2+\lambda_1\lambda_i+c.
\end{align}
\end{lemma}
\begin{proof}
Substituting $H=-2\lambda_1/n$ into the first equation of (2.6), and using
(2.7), (3.6), (3.3) and (3.5),  we get (3.11). By putting $i=1$ in
(3.8), combining this with (3.9) gives (3.12).

Next, we will prove equation (3.13).

 For $j=1$ and $i\neq1$ in (3.8),
by (3.6) we have $\omega_{1i}^1=0$ $(i\neq1)$. Combining this with
(3.7), we have
\begin{eqnarray}
\omega_{11}^i=0\quad \mathrm {for} ~~1\leq i \leq n.
\end{eqnarray}
For $j=1$, and $k, i\neq1$ in (3.9) we have
\begin{eqnarray*}
(\lambda_i-\lambda_1)\omega_{ki}^1=(\lambda_k-\lambda_1)\omega_{ik}^1,
\end{eqnarray*}
which together with (3.10) yields
\begin{eqnarray}
\omega_{ki}^1=0,~~~ k\neq i,~~~ \mathrm {if}~~~
\lambda_k\neq\lambda_i.
\end{eqnarray}
For $i\neq j$ and $2\leq i, j\leq n$, if $\lambda_i=\lambda_j$, then
by putting $k=1$ in (3.9) we have
\begin{eqnarray*}
(\lambda_1-\lambda_i)\omega_{i1}^j=0,
\end{eqnarray*}
which together with Lemma 3.1, (3.15) and (3.7) yields
\begin{eqnarray}
\omega_{i1}^j=0,~~~ i\neq j,~~~\mathrm {and}~~~2\leq i, j\leq n.
\end{eqnarray}
From the Gauss equation and (3.1), we have $\langle R(e_1,e_i)e_1,
e_i\rangle=-\lambda_1\lambda_i-c$. On the other hand, the Gauss
curvature tensor $R$ is defined by $
R(X,Y)Z=\nabla_X\nabla_YZ-\nabla_Y\nabla_XZ-\nabla_{[X,Y]}Z$. Using
(3.14), (3.16) and (3.7), a direct computation gives
\begin{eqnarray*}
\langle R(e_1,e_i)e_1,
e_i\rangle=-e_1(\omega_{ii}^1)+(\omega_{ii}^1)^2.
\end{eqnarray*}
Therefore, we obtain differential equation (3.13), which completes
the proof of Lemma 3.2.
\end{proof}
Consider an integral curve of $e_1$ passing through $p=\gamma(t_0)$
as $\gamma(t)$, $t\in I$. Since $e_i(\lambda_1)=0$ for $2\leq i \leq
n$ and $e_1(\lambda_1)\neq0$, it is easy to show that there exists a
local chart $(U; t=x^1,x^2,\ldots, x^m)$ around $p$, such that
$\lambda_1(t, x^2,\ldots, x^m)=\lambda_1(t)$ on the whole
neighborhood of $p$.

In the following, we begin our arguments under the assumption that
the scalar curvature $R$ is always constant. The following system of
algebraic equations is important for us to proceed further.
\begin{lemma}
Assume that $R$ is constant. We have
\begin{align}
\sum_{i=2}^n(\omega_{ii}^1)^k=f_k(t),~~\mathrm {for}~ k=1,\ldots, 5,
\end{align}
where $f_k(t)$ are some smooth real-valued functions with respect to
$t$.
\end{lemma}
\begin{proof}
Since $e_1(\lambda_1)\neq0$, $\lambda_1=\lambda_1(t)$ and $R$ is
constant, (3.11) becomes
\begin{align}
\sum_{i=2}^n\omega_{ii}^1=f_1(t),
\end{align}
where
\begin{align*}
f_1(t)=\frac{e_1e_1(\lambda_1)-\lambda_1\big(n(n-2)c+4\lambda_1^2-R\big)}{e_1(\lambda_1)}.
\end{align*}
Taking the sum of   (3.13) and (3.12) for $i$ and taking into
account (3.2) and (3.18) respectively, we have
\begin{align}
\sum_{i=2}^n\big(\omega_{ii}^1\big)^2&=f_2(t),\\
\sum_{i=2}^n\lambda_i\omega_{ii}^1&=g_1(t),
\end{align}
where $f_2=3\lambda_1^2-(n-1)c+e_1(f_1)$ and
$g_1(t)=\lambda_1f_1-3e_1(\lambda_1)$.

Multiplying $\omega_{ii}^1$ on both sides of equation (3.13), we
have
\begin{align*}
\frac{1}{2}e_1\big((\omega_{ii}^1)^2\big)&=(\omega_{ii}^1)^3+\lambda_1\lambda_i\omega_{ii}^1+c\omega_{ii}^1.
\end{align*}

Taking the sum of the above equation and using (3.18)-(3.20), we
obtain
\begin{eqnarray}
\sum_{i=2}^n\big(\omega_{ii}^1\big)^3=f_3(t),
\end{eqnarray}
where $f_3=\frac{1}{2}e_1(f_2)-\lambda_1 g_1-cf_1$.

Differentiating (3.20) with respect to $e_1$ and using (3.12) and
(3.13), we have
\begin{eqnarray}
e_1(g_1)=2\sum_{i=2}^n\lambda_i\big(\omega_{ii}^1\big)^2+
\lambda_1\sum_{i=2}^n\lambda_i^2+c\sum_{i=2}^n\lambda_i-\lambda_1\sum_{i=2}^n\big(\omega_{ii}^1\big)^2.
\end{eqnarray}
Hence, from (3.2), (3.5) and (3.19) that (3.22) yields
\begin{eqnarray}
\sum_{i=2}^n\lambda_i\big(\omega_{ii}^1\big)^2=g_2(t),
\end{eqnarray}
where $g_2=\frac{1}{2}\big\{e_1(g_1)-
\lambda_1\big(n(n-1)c-R+3\lambda_1^2\big)+3c\lambda_1+\lambda_1f_2\big\}.$

Multiplying $(\omega_{ii}^1)^2$ on both sides of equation (3.13), we
have
\begin{align*}
\frac{1}{3}e_1\big((\omega_{ii}^1)^3\big)&=(\omega_{ii}^1)^4+\lambda_1\lambda_i(\omega_{ii}^1)^2+c(\omega_{ii}^1)^2.
\end{align*}
Taking the sum of the above equation for $i$ and applying (3.19),
(3.21) and (3.23), we obtain
\begin{eqnarray}
\sum_{i=2}^n\big(\omega_{ii}^1\big)^4=f_4(t),
\end{eqnarray}
where $f_4=\frac{1}{3}e_1(f_3)-\lambda_1 g_2-cf_2$.

Multiplying $\lambda_i$ on both sides of equation (3.12) gives
\begin{align*}
\lambda_i^2\omega_{ii}^1=\frac{1}{2}e_1(\lambda_i^2)+\lambda_1\lambda_i\omega_{ii}^1,
\end{align*}
which together with (3.5) and (3.20) yields
\begin{eqnarray}
\sum_{i=2}^n\lambda_i^2\omega_{ii}^1=g_3(t),
\end{eqnarray}
where $g_3=3\lambda_1e_1(\lambda_1)+\lambda_1g_1$.

Differentiating (3.23) with respect to $e_1$ and using
(3.12)-(3.13), we have
\begin{align}
e_1(g_2)=3\sum_{i=2}^n\lambda_i\big(\omega_{ii}^1\big)^3-
\lambda_1\sum_{i=2}^n\big(\omega_{ii}^1\big)^3+2\lambda_1\sum_{i=2}^n\lambda_i^2\omega_{ii}^1+2c\sum_{i=2}^n\lambda_i\omega_{ii}^1.
\end{align}
Substituting (3.20), (3.21) and (3.25) into (3.26) gives
\begin{eqnarray}
\sum_{i=2}^n\lambda_i\big(\omega_{ii}^1\big)^3=g_4(t),
\end{eqnarray}
where $g_4=\frac{1}{3}\big
(e_1(g_2)+\lambda_1f_3-2\lambda_1g_3-2cg_1\big).$

Multiplying $(\omega_{ii}^1)^3$ on both sides of equation (3.13), we
have
\begin{align*}
\frac{1}{4}e_1\big((\omega_{ii}^1)^4\big)&=(\omega_{ii}^1)^5+\lambda_1\lambda_i(\omega_{ii}^1)^3+c(\omega_{ii}^1)^3.
\end{align*}
After taking the sum of the above equation for $i$, using
(3.21), (3.24) and (3.27) we have
\begin{eqnarray}
\sum_{i=2}^n\big(\omega_{ii}^1\big)^5=f_5(t),
\end{eqnarray}
where $f_5=\frac{1}{4}e_1(f_4)-\lambda_1 g_4-cf_3$.

At this moment, the proof of Lemma 3.3 has been completed.
\end{proof}
\begin{lemma}
Assume that $R$ is constant. If the number $m$ of distinct principal
curvatures satisfies $m\leq6$, then $e_i(\lambda_j)=0$ for $2\leq i,
j\leq n$, i.e. all principal curvature $\lambda_i$ depend only on
one variable $t$.
\end{lemma}
\begin{proof}
Since the number $m$ of distinct principal curvatures satisfies
$m\leq6$, there are at most five distinct principal curvatures for
$\lambda_i$ $(2\leq i\leq n)$ except $\lambda_1$. It follows easily from (3.12) and (3.13) that
\begin{align*}
\lambda_i\neq \lambda_j\quad \Leftrightarrow\quad \omega_{ii}^1\neq \omega_{jj}^1.
\end{align*}

We now distinguish the following two cases:

{\bf Case} A. ~Suppose that $m=6$. We denote by $\widetilde{\lambda}_{i}$ the
five distinct principal curvatures with the corresponding
multiplicities $n_i$ for $1\leq i\leq5$. Note that here $n_i$ are
positive integers and $\sum_{i=1}^5n_i=n-1$ (see Lemma 3.1).
According to (3.12), let
$$u_i:=\frac{e_1(\widetilde{\lambda}_{i})}{\widetilde{\lambda}_{i}-\lambda_1}.$$
Thus, $u_i$ are mutually different for $1\leq i\leq 5$.

In this case, the system of polynomial equations (3.17) becomes
\begin{align}
\begin{cases}
n_1u_1+n_2u_2+n_3u_3+n_4u_4+n_5u_5=f_1,\\
n_1u^2_1+n_2u^2_2+n_3u^2_3+n_4u^2_4+n_5u^2_5=f_2,\\
n_1u^3_1+n_2u^3_2+n_3u^3_3+n_4u^3_4+n_5u^3_5=f_3,\\
n_1u^4_1+n_2u^4_2+n_3u^4_3+n_4u^4_4+n_5u^4_5=f_4,\\
n_1u^5_1+n_2u^5_2+n_3u^5_3+n_4u^5_4+n_5u^5_5=f_5.
\end{cases}
\end{align}
Since $e_i(f_1)=0$ for $2\leq i\leq n$, differentiating both sides
of equations in (3.29) with respect to $e_i$ $(2\leq i\leq n)$, we
obtain
\begin{align}
\begin{cases}
n_1 e_i(u_1)+n_2e_i(u_2)+n_3e_i(u_3)+n_4e_i(u_4)+n_5e_i(u_5)=0,\\
n_1u_1e_i(u_1)+n_2u_2e_i(u_2)+n_3u_3e_i(u_3)+n_4u_4e_i((u_4)+n_5u_5e_i(u_5)=0,\\
n_1u^2_1e_i(u_1)+n_2u^2_2e_i(u_2)+n_3u^2_3e_i(u_3)+n_4u^2_4e_i(u_4)+n_5u^2_5e_i(u_5)=0,\\
n_1u^3_1e_i(u_1)+n_2u^3_2e_i(u_2)+n_3u^3_3e_i(u_3)+n_4u^3_4e_i(u_4)+n_5u^3_5e_i(u_5)=0,\\
n_1u^4_1e_i(u_1)+n_2u^4_2e_i(u_2)+n_3u^4_3e_i(u_3)+n_4u^4_4e_i(u_4)+n_5u^4_5e_i(u_5)=0.
\end{cases}
\end{align}
Now consider this system of five linear equations with five unknowns
$e_i(u_k)$ for $ 1\leq k\leq5$.

According to Cramer's rule in linear algebra, for any $k$,
$e_i(u_k)\equiv 0$ holds true if and only if the determinant of the
coefficient matrix of (3.30) is not vanishing, i.e.
\begin{eqnarray}
\left | \begin{array}{ccccc} 1&1&1&1&1\\
u_1&u_2&u_3&u_4&u_5\\u^2_1&u^2_2&u^2_3&u^2_4&u^2_5
\\u^3_1&u^3_2&u^3_3&u^3_4&u^3_5\\
u^4_1&u^4_2&u^4_3&u^4_4&u^4_5
\end{array} \right |\neq 0.
\end{eqnarray}
We note that the determinant in (3.31) is the famous Vandermonde
determinant with order 5 and hence
\begin{eqnarray}
\left | \begin{array}{ccccc} 1&1&1&1&1\\
u_1&u_2&u_3&u_4&u_5\\u^2_1&u^2_2&u^2_3&u^2_4&u^2_5
\\u^3_1&u^3_2&u^3_3&u^3_4&u^3_5\\
u^4_1&u^4_2&u^4_3&u^4_4&u^4_5
\end{array} \right |=\prod_{1\leq j<i\leq 5}(u_i-u_j).
\end{eqnarray}
Since $u_i$ are mutually different for $1\leq i\leq 5$, (3.32) implies
that (3.31) holds true identically. Hence, we have $e_i(u_k)=0$ for any
$1\leq k\leq 5$ and $2\leq i\leq n$.

Therefore, by using  $e_i(u_k)=0$ and
\begin{eqnarray}
e_ie_1(u_k)-e_1e_i(u_k)=[e_i, e_1](u_k)=\sum_{j=2}^n(\omega_{i1}^j-\omega_{1i}^j)e_j(u_k),\nonumber
\end{eqnarray}
we get
\begin{eqnarray}
e_ie_1(u_k)=0.\nonumber
\end{eqnarray}

Noting that with the notation $u_k$, (3.13) becomes
$$e_1(u_k)=(u_k)^2+\lambda_1\lambda_k+c.$$

Differentiating the above equation with respect to $e_i$, by taking into account $e_i(u_k)=0$ and $e_ie_1(u_k)=0$ we derive
$$e_i(\lambda_k)=0$$
for any $1\leq k\leq 5$ and $2\leq i\leq
n$.

{\bf Case} B. ~Suppose $m\leq5$. Denote by $\widetilde{\lambda}_{i}$ the
distinct principal curvatures with the corresponding
multiplicities $n_i$ for $1\leq i\leq4$.
Then the number of different $u_i$ is less than or equal to four.
In the case that four ones of $u_i$ are mutually different,
it is needed only to consider the system (3.17)
for $k=1, 2, 3, 4$. A similar discussion as in Case A could yield the conclusion.
If three ones or less of $u_i$ are mutually different,
then the conclusion follows by some similar arguments as above.

Thus, we conclude Lemma 3.4.
\end{proof}
\begin{lemma}
For arbitrary three distinct principal curvatures $\lambda_i$,
$\lambda_j$ and $\lambda_k$ $(2\leq i, j, k\leq n)$, we have the following relations:
\begin{align}
&\omega_{ij}^k(\lambda_j-\lambda_k)=\omega_{ji}^k(\lambda_i-\lambda_k)=\omega_{kj}^i(\lambda_j-\lambda_i),\\
&\omega_{ij}^k\omega_{ji}^k+\omega_{jk}^i\omega_{kj}^i+\omega_{ik}^j\omega_{ki}^j=0,\\
&\omega_{ij}^k(\omega_{jj}^1-\omega_{kk}^1)=\omega_{ji}^k(\omega_{ii}^1-\omega_{kk}^1)=
\omega_{kj}^i(\omega_{jj}^1-\omega_{ii}^1).
\end{align}
\end{lemma}
\begin{proof}
We recall in the beginning part of this section that the number $m$
of distinct principal curvatures satisfies $m\geq4$. Hence, by
taking into account the second expression of (3.7) and (3.9) for
three distinct principal curvatures $\lambda_i$, $\lambda_j$ and
$\lambda_k$ $(2\leq i, j, k\leq n)$, we obtain (3.33) and (3.34)
immediately.

Let us consider (3.35). It follows from the Gauss equation that
\begin{eqnarray*}
\langle R(e_i,e_j)e_k, e_1\rangle=0.
\end{eqnarray*}
Moreover, since $\omega_{ij}^1=0$ for $i\neq j$ from (3.7) and
(3.16), from the definition of the curvature tensor we have
\begin{align}
\omega_{ij}^k(\omega_{jj}^1-\omega_{kk}^1)=\omega_{ji}^k(\omega_{ii}^1-\omega_{kk}^1).
\end{align}
Similarly, by considering $\langle R(e_j,e_k)e_i, e_1\rangle=0$ one also has
\begin{align*}
\omega_{jk}^i(\omega_{kk}^1-\omega_{ii}^1)=\omega_{kj}^i(\omega_{jj}^1-\omega_{ii}^1),
\end{align*}
which together with (3.7) and (3.36) gives (3.35).
\end{proof}
\begin{lemma}
Under the assumptions as above, we have
\begin{align}
\omega_{ii}^1\omega_{jj}^1-\sum_{k=2, ~k\neq l_{(i,j)}}^{n}2\omega_{ij}^k\omega_{ji}^k=-\lambda_i\lambda_j-c,
\quad \mathrm{for}~~\lambda_i\neq\lambda_j,
\end{align}
where $l_{(i,j)}$ stands for the indexes satisfying $\lambda_{l_{(i,j)}}=\lambda_i$ or $\lambda_j$.
\end{lemma}
\begin{proof}
In the following, we consider the case that the number $m$ of distinct principal curvatures is $6$.

Without loss of generality, except $\lambda_1$, we assume that
$\lambda_p, \lambda_q, \lambda_r, \lambda_u, \lambda_v$ are the five distinct
principal curvatures in sequence with the corresponding multiplicities $n_1, n_2, n_3, n_4, n_5$ respectively, i.e.
\begin{eqnarray*}
\lambda_1,~\underbrace{\lambda_p, \ldots, \lambda_p}_{n_1}
,~\underbrace{\lambda_q, \ldots, \lambda_q}_{n_2},
~\underbrace{\lambda_r, \ldots, \lambda_r}_{n_3},
~\underbrace{\lambda_u, \ldots, \lambda_u}_{n_4},
~\underbrace{\lambda_v, \ldots, \lambda_v}_{n_5}.
\end{eqnarray*}
We now compute $\langle R(e_p,e_q)e_p, e_q\rangle$. On one hand, it
follows from the Gauss equation and (3.1) that
\begin{eqnarray}
\langle R(e_p,e_q)e_p, e_q\rangle=-\lambda_p\lambda_q-c.
\end{eqnarray}
On the other hand, since
\begin{eqnarray*}
&&\nabla_{e_p}\nabla_{e_q}e_p=\sum_{k=1}^{n}e_p\big(\omega_{qp}^k\big)e_k+
\sum_{k=1}^{n}\omega_{qp}^k\sum_{l=1}^{n}\omega_{pk}^le_l,\\
&&\nabla_{e_q}\nabla_{e_p}e_p=\sum_{k=1}^{n}e_q\big(\omega_{pp}^k\big)e_k+
\sum_{k=1}^{n}\omega_{pp}^k\sum_{l=1}^{n}\omega_{qk}^le_l,\\
&&\nabla_{[e_p,e_q]}e_p=\sum_{k=1}^{n}\big(\omega_{pq}^k-\omega_{qp}^k\big)\sum_{l=1}^{n}\omega_{kp}^le_l,
\end{eqnarray*}
it follows that
\begin{align}
\langle R(e_p,e_q)e_p, e_q\rangle&=e_p\big(\omega_{qp}^q\big)+\sum_{k=1}^{n}\omega_{qp}^k\omega_{pk}^q
-e_q\big(\omega_{pp}^q\big)\\
&-\sum_{k=1}^{n}\omega_{pp}^k\omega_{qk}^q-
\sum_{k=1}^{n}\big(\omega_{pq}^k-\omega_{qp}^k\big)\omega_{kp}^q.\nonumber
\end{align}
Since $\lambda_p\neq\lambda_q$, from (3.8), (3.7) and Lemma 3.4
we have
\begin{align}
\omega_{qp}^q=\omega_{qq}^p=\omega_{pp}^q=0,~~\mathrm{and}~~\sum_{k=2}^{n}\omega_{pp}^k\omega_{qk}^q=0.
\end{align}
Moreover, if $2\leq k\leq n_1+1$, then $\lambda_k=\lambda_p$, by the second expression of (3.7) and (3.9) we get
\begin{align*}
(\lambda_p-\lambda_k)\omega_{qp}^k=(\lambda_q-\lambda_k)\omega_{pq}^k,~~\mathrm{and}~~
(\lambda_k-\lambda_q)\omega_{pk}^q=(\lambda_p-\lambda_q)\omega_{kp}^q,
\end{align*}
which imply that
\begin{align}
\omega_{pq}^k=\omega_{pk}^q=\omega_{kp}^q=0.
\end{align}
Similarly, if $n_1+2\leq k\leq n_1+n_2+1$, we also have
\begin{align}
\omega_{pq}^k=\omega_{pk}^q=\omega_{kp}^q=0.
\end{align}
Hence, by taking into account (3.40)-(3.42), (3.39) becomes
\begin{align*}
\langle R(e_p,e_q)e_p, e_q\rangle=\omega_{pp}^1\omega_{qq}^1+ \sum_{k=n_1+n_2+2}^{n}\Big\{\omega_{qp}^k\omega_{pk}^q
-\big(\omega_{pq}^k-\omega_{qp}^k\big)\omega_{kp}^q\Big \},\nonumber
\end{align*}
which together with (3.38), (3.7) and (3.34) gives
\begin{align}
\omega_{pp}^1\omega_{qq}^1-\sum_{k=n_1+n_2+2}^{n}2\omega_{pq}^k\omega_{qp}^k=-\lambda_p\lambda_q-c.
\end{align}
Similarly, we could deduce other equations for different pairs
$\omega_{pp}^1\omega_{rr}^1, \omega_{pp}^1\omega_{uu}^1, \cdots$.
Hence we get equation (3.37).

In the case that the number $m$ of distinct principal curvatures
satisfies $m=4$, or $5$, a very similar argument gives (3.37) as well.
\end{proof}

\section{Proof of Theorem 1.1}
Assume that the mean curvature $H$ is not constant.

Differentiating (3.2) with respect to $e_1$ and using (3.12)-(3.13), we obtain
\begin{align}
3e_1(\lambda_1)=\sum_{i=2}^n(\lambda_1-\lambda_i)\omega_{ii}^1.
\end{align}
Following the previous section, we only deal with the case that the
number of distinct principal curvatures is $6$, i.e. $m=6$. In fact,
the proofs for the cases that $m=5$, $4$ are very similar, so we
omit it here without loss of generality.

According to Lemma 3.5, we consider the following cases:

{\bf Case} A. $\omega_{pq}^r\neq0, \omega_{pq}^u\neq0$, and
$\omega_{pq}^v\neq0$. Since $\lambda_p, \lambda_q, \lambda_r,
\lambda_u,\lambda_v$ are mutually different, equations (3.33) and
(3.35) reduce to
\begin{align*}
\frac{\omega_{pp}^1-\omega_{qq}^1}{\lambda_p-\lambda_q}
&=\frac{\omega_{pp}^1-\omega_{rr}^1}{\lambda_p-\lambda_r}
=\frac{\omega_{qq}^1-\omega_{rr}^1}{\lambda_q-\lambda_r}\\
&=\frac{\omega_{pp}^1-\omega_{uu}^1}{\lambda_p-\lambda_u}
=\frac{\omega_{qq}^1-\omega_{uu}^1}{\lambda_q-\lambda_u}\nonumber\\
&=\frac{\omega_{pp}^1-\omega_{vv}^1}{\lambda_p-\lambda_v}
=\frac{\omega_{qq}^1-\omega_{vv}^1}{\lambda_q-\lambda_v}.\nonumber\
\end{align*}
Thus, there exist two smooth functions $\varphi$ and $\psi$ depending on $t$ such that
\begin{align}
\omega_{ii}^1=\varphi \lambda_i+\psi.
\end{align}
Differentiating with respect to $e_1$ on both sides of equation (4.2), and using (3.12) and (3.13) we get
\begin{align}
&e_1(\varphi)=\lambda_1(\varphi^2+1)+\varphi\psi,\\
&e_1(\psi)=\psi(\lambda_1\varphi+\psi)+c.
\end{align}
Taking into account (4.2), and using (3.2), (3.5) one has
\begin{align*}
\sum_{i=2}^n\omega_{ii}^1=-3\lambda_1\varphi+(n-1)\psi,
\end{align*}
 and (4.1) and (3.11) respectively become
\begin{align}
&3e_1(\lambda_1)=\big(R-n(n-1)c-6\lambda_1^2\big)\varphi+(n+2)\lambda_1\psi,\\
&e_1e_1(\lambda_1)=e_1(\lambda_1)(-3\lambda_1\varphi+(n-1)\psi)+\lambda_1\big(n(n-2)c-R+4\lambda_1^2\big).
\end{align}
Differentiating (4.5) with respect to $e_1$, we may eliminate
$e_1e_1(\lambda_1)$ by (4.6). Using (4.3), (4.4) and (4.6) we have
\begin{align}
3(n-4)e_1(\lambda_1)\psi=\lambda_1\big(6R-(4n^2-12n-3)c-27\lambda_1^2\big).
\end{align}
Note here that $n>4$ since the number of distinct principal curvatures is six.

Eliminating $e_1(\lambda_1)$ between (4.5) and (4.7) gives
\begin{align}
(n-4)&\Big\{\big(R-n(n-1)c-6\lambda_1^2\big)\varphi\psi+(n+2)\lambda_1\psi^2\Big\}\\
&=\lambda_1\big(6R-(4n^2-12n-3)c-27\lambda_1^2\big).\nonumber
\end{align}
Moreover, differentiating (4.7) with respect to $e_1$, by (4.4), (4.6), (4.7) we have
\begin{align}
\big(432\lambda_1^4+a_1\lambda_1^2+
a_2\big)\varphi+\big\{-54(n+3)\lambda_1^3+a_3\lambda_1\big\}\psi=12(n-4)\lambda_1^3+a_4\lambda_1,
\end{align}
where
\begin{align*}
&a_1=(97n^2-111n+60)c-105R, \\
&a_2=\big((4n^2-9n+9)c-6R\big)\big(n(n-1)c-R\big),\\
&a_3=12R-(4n^2-6n+21)c,\\
& a_4=3n(n-4)(n-2)c.
\end{align*}
Differentiating (4.9) with respect to $e_1$ and  using (4.3)-(4.4),
we get
\begin{align}
&\big(1728\lambda_1^3+2a_1\lambda_1\big)\varphi
e_1(\lambda_1)+\big(432\lambda_1^4+a_1\lambda_1^2+
a_2\big)\big\{\lambda_1(\varphi^2+1)+\varphi\psi\big\}\nonumber\\
&+\big\{-162(n+3)\lambda_1^2+a_3\big\}\psi e_1(\lambda_1)+
\big\{-54(n+3)\lambda_1^3+a_3\lambda_1\big\}\big\{\psi(\lambda_1\varphi+\psi)+c
\big\} \nonumber\\
&=(36(n-4)\lambda_1^2+a_4)e_1(\lambda_1).\nonumber
\end{align}
Multiplying $3(n-4)$ on both sides of the above equation and using
(4.5) and (4.7) we have
\begin{align}
&(n-4)\big(1728\lambda_1^3+2a_1\lambda_1\big)\varphi
\big\{\big(R-n(n-1)c-6\lambda_1^2\big)\varphi+(n+2)\lambda_1\psi\big\}\\
&+3(n-4)\big(432\lambda_1^4+a_1\lambda_1^2+
a_2\big)\big\{\lambda_1(\varphi^2+1)+\varphi\psi\big\}\nonumber\\
&+\lambda_1\big\{-162(n+3)\lambda_1^2+a_3\big\}\big\{6R-(4n^2-12n-3)c-27\lambda_1^2\big\}\nonumber\\
&+3(n-4)
\big\{-54(n+3)\lambda_1^3+a_3\lambda_1\big\}\big\{\psi(\lambda_1\varphi+\psi)+c
\big\} \nonumber\\
&=(n-4)\big(36(n-4)\lambda_1^2+a_4\big)\big\{\big(R-n(n-1)c-6\lambda_1^2\big)\varphi+(n+2)\lambda_1\psi\big\}.\nonumber
\end{align}
Note that equation (4.10) could be rewritten as
\begin{align}
q_1(\lambda_1)\varphi^2+q_2(\lambda_1)\varphi\psi+q_3(\lambda_1)\psi^2+
q_4(\lambda_1)\varphi+q_5(\lambda_1)\psi+q_6(\lambda_1)=0,
\end{align}
where $q_i$ are non-trivial polynomials concerning function
$\lambda_1$ and given by:
\begin{equation}
\begin{cases}
q_1=(n-4)\big(1728\lambda_1^3+2a_1\lambda_1\big)
\big(R-n(n-1)c-6\lambda_1^2\big)\\
~~\quad\quad +3(n-4)\big(432\lambda_1^4+a_1\lambda_1^2+
a_2\big)\lambda_1,\\
q_2=(n-4)(n+2)\lambda_1\big(1728\lambda_1^3+2a_1\lambda_1\big)\\
~~\quad\quad+3(n-4)\big(432\lambda_1^4+a_1\lambda_1^2+
a_2\big)\\
~~\quad\quad+3(n-4)\big\{-54(n+3)\lambda_1^3+a_3\lambda_1\big\}\lambda_1,\\
q_3=3(n-4)\big\{-54(n+3)\lambda_1^3+a_3\lambda_1\big\},\\
q_4=(n-4)\big(36(n-4)\lambda_1^2+a_4\big)\big(R-n(n-1)c-6\lambda_1^2\big),\\
q_5=-(n-4)(n+2)\big(36(n-4)\lambda_1^2+a_4\big)\lambda_1,\\
q_6=-3(n-4)\big(432\lambda_1^4+a_1\lambda_1^2+a_2\big)\lambda_1\\
~~\quad\quad+\lambda_1\big(-162(n+3)\lambda_1^2+a_3\big)\big\{6R-(4n^2-12n-3)c-27\lambda_1^2\big\}\\
~~\quad\quad+ 3c(n-4)
\big\{-54(n+3)\lambda_1^3+a_3\lambda_1\big\}.
\end{cases}
\end{equation}
In the same manner, (4.8) and (4.9) could be also rewritten
respectively as:
\begin{align}
&p_1(\lambda_1)\varphi\psi+p_2(\lambda_1)\psi^2=p_3(\lambda_1),\\
&h_1(\lambda_1)\varphi+h_2(\lambda_1)\psi=h_3(\lambda_1),
\end{align}
where $p_i, h_i$ $(i=1, 2)$ are polynomials concerning function
$\lambda_1$ and given by
\begin{align}
\begin{cases}
p_1=(n-4)\big(R-n(n-1)c-6\lambda_1^2\big),\\
p_2=(n-4)(n+2)\lambda_1,\\
p_3=\lambda_1\big(6R-(4n^2-12n-3)c-27\lambda_1^2\big),\\
h_1=432\lambda_1^4+a_1\lambda_1^2+
a_2,\\
h_2=-54(n+3)\lambda_1^3+a_3\lambda_1,\\
h_3=12(n-4)\lambda_1^3+a_4\lambda_1.
\end{cases}
\end{align}
Multiplying $h_1^2$ on both sides of the equation (4.11), by taking
into account (4.14) we may eliminate $\varphi$ and get
\begin{align}
P_1\psi^2+P_2\psi=P_3,
\end{align}
where
\begin{align}
\begin{cases}
P_1=q_1h_2^2-q_2h_1h_2+q_3h_1^2,\\
P_2=-2q_1h_2h_3+q_2h_1h_3-q_4h_1h_2+q_5h_1^2,\\
P_3=-q_1h_3^2-q_4h_1h_3-q_6h_1^2.
\end{cases}
\end{align}
Similarly, eliminating $\varphi$ in (4.13) by using (4.14) yields
\begin{align}
Q_1\psi^2+Q_2\psi=Q_3,
\end{align}
where
\begin{align}
\begin{cases}
Q_1=p_2h_1-p_1h_2,\\
Q_2=p_1h_3,\\
Q_3=p_3h_1.
\end{cases}
\end{align}
Moreover, multiplying $Q_1$ and $P_1$ on both sides of the equations
(4.16) and (4.18) respectively, after eliminating the `$\psi^2$'
part we obtain
\begin{align}
(P_2Q_1-P_1Q_2)\psi=P_3Q_1-P_1Q_3.
\end{align}
Multiplying $P_1\psi$ on (4.20) and then combining this with (4.16)
give
\begin{align}
\big\{P_1(P_3Q_1-P_1Q_3)+P_2(P_2Q_1-P_1Q_2)\big\}\psi=P_3(P_2Q_1-P_1Q_2).
\end{align}
At last, after eliminating $\psi$ between (4.20) and(4.21) we get
\begin{align}
P_1(P_3Q_1-P_1Q_3)^2+P_2(P_2Q_1-P_1Q_2)(P_3Q_1-P_1Q_3)
\\=P_3(P_2Q_1-P_1Q_2)^2.\nonumber
\end{align}
We observe from (4.12), (4.15), (4.17) and (4.19) that both $P_i$ and $Q_i$ $(1\leq i\leq 3)$ are polynomials
concerning $\lambda_1$ with constant coefficients. Hence, it follows that
\begin{align}
&P_1=-10077696(n-4)(n+3)(n-1)\lambda_1^{11}+ \cdots,\nonumber\\
&P_2=-839808(n-4)^2(11n+5)\lambda_1^{11}+ \cdots,\nonumber\\
&P_3=-69984(19n+113)\lambda_1^{13}+\cdots,\nonumber\\
&Q_1=108(n-4)(n-1)\lambda_1^{5}+\cdots,\nonumber\\
&Q_2=-72(n-4)^2\lambda_1^{5}+\cdots,\nonumber\\
&Q_3=-11664\lambda_1^{7}+\cdots,\nonumber
\end{align}
where we only need to write the highest order terms of $\lambda_1$.

By substituting $P_i$ and $Q_i$ into equation (4.22), we get a
polynomial equation concerning $\lambda_1$ with constant
coefficients $c_i=c_i(n, c, R)$:
\begin{align}
\sum_{i=0}^{47} c_i\lambda_1^i=0,
\end{align}
where the coefficient $c_{47}$ of the highest order term  satisfies
\begin{align}
c_{47}&=-10077696(n-4)^2(n+3)(n-1)^2\big[69984\times
108(19n+113)\nonumber\\
&\quad+10077696\times 11664(n+3)\big]^2\neq0.\nonumber
\end{align}
Therefore, $\lambda_1$ has to be constant and $H=-2\lambda_1/n$ is a
constant, which is a contradiction.

{\bf Case} B. $\omega_{pq}^r\neq0, \omega_{pq}^u\neq 0$, and
$\omega_{ij}^k=0$ for all other distinct triplets $\{i, j, k\}$ and
distinct principal curvatures $\lambda_i, \lambda_j, \lambda_k$.
Then, (3.37) implies that
\begin{align}
\omega_{pp}^1\omega_{vv}^1=-\lambda_p\lambda_v-c,\\
\omega_{qq}^1\omega_{vv}^1=-\lambda_q\lambda_v-c,\\
\omega_{rr}^1\omega_{vv}^1=-\lambda_r\lambda_v-c,\nonumber\\
\omega_{uu}^1\omega_{vv}^1=-\lambda_u\lambda_v-c.\nonumber
\end{align}
Similar to Case A, since $\omega_{pq}^r\neq0, \omega_{pq}^u\neq 0$,
(3.33) and (3.35) imply that
\begin{align}
\omega_{ii}^1=\varphi \lambda_i+\psi,\quad\mathrm{for}~~i=p, q, r,
u.
\end{align}
where $\varphi$ and $\psi$ satisfy the differential equations (4.3)
and (4.4).

Substituting (4.26) into (4.24) and (4.25), we obtain
\begin{align}
&\omega_{vv}^1=-\frac{1}{\varphi}\lambda_v,\\
&\lambda_v\psi=c\varphi,
\end{align}
which means that $\omega_{vv}^1$ and $\lambda_v$ are determined
completely by $\varphi$ and $\psi$.

 Substitute (4.26)-(4.28) into (4.1), and then
differentiate it with respect to $e_1$. By using (4.3), (4.4) and
(3.11), a similar discussion as Case A could give a polynomial
concerning function $\lambda_1$ with constant coefficients. Hence,
$\lambda_1$ has to be constant, which yields a contradiction as
well.

{\bf Case} C.  $\omega_{pq}^r\neq0$ (or $\omega_{pq}^r=0$), and all the $\omega_{ij}^k=0$ for distinct
triplets $\{i, j, k\}$ and distinct principal curvatures $\lambda_i, \lambda_j, \lambda_k$. Then, (3.37) implies that
\begin{align}
&\omega_{pp}^1\omega_{uu}^1=-\lambda_p\lambda_u-c,\quad \omega_{pp}^1\omega_{vv}^1=-\lambda_p\lambda_v-c,\\
&\omega_{qq}^1\omega_{uu}^1=-\lambda_q\lambda_u-c,\quad \omega_{qq}^1\omega_{vv}^1=-\lambda_q\lambda_v-c,\\
&\omega_{rr}^1\omega_{uu}^1=-\lambda_r\lambda_u-c,\quad \omega_{rr}^1\omega_{vv}^1=-\lambda_r\lambda_v-c,\\
&\omega_{uu}^1\omega_{vv}^1=-\lambda_u\lambda_v-c.
\end{align}
We first consider $\lambda_i\neq0$ for $i=p, q, r, u, v$.
Consequently, (4.29)-(4.32) reduce to
\begin{align*}
&\frac{\omega_{pp}^1}{\lambda_p}=\frac{\omega_{qq}^1}{\lambda_q}=\frac{\omega_{rr}^1}{\lambda_r}=-
\frac{\lambda_u-\lambda_v}{\omega_{uu}^1-\omega_{vv}^1},\\
&\frac{\omega_{uu}^1}{\lambda_u}=\frac{\omega_{vv}^1}{\lambda_v}=-
\frac{\lambda_p-\lambda_q}{\omega_{pp}^1-\omega_{qq}^1},
\end{align*}
and hence
\begin{align}
&\frac{\omega_{pp}^1}{\lambda_p}=\frac{\omega_{qq}^1}{\lambda_q}=\frac{\omega_{rr}^1}{\lambda_r}=\varphi,\\
&\frac{\omega_{uu}^1}{\lambda_u}=\frac{\omega_{vv}^1}{\lambda_v}=\psi
\end{align}
for two functions $\varphi$ and $\psi$.

Substituting (4.33) and (4.34) back to (4.29) gives
\begin{align*}
(1+\varphi\psi)\lambda_p\lambda_u=-c,\\
(1+\varphi\psi)\lambda_p\lambda_v=-c,
\end{align*}
which imply that $\lambda_u=\lambda_v$. This is impossible.

If $\lambda_p=0$, then (3.12) and (4.29) imply that
$\omega_{pp}^1=0$ and $c=0$. Then (4.30) and (4.31) yield
\begin{align}
\frac{\omega_{uu}^1}{\lambda_u}=\frac{\omega_{vv}^1}{\lambda_v}=\gamma
\end{align}
for some function $\gamma$. However, combining (4.35) with (4.32)
gives $\gamma^2=-1$. Hence it is a contradiction.

At last, we consider $\lambda_u=0$. Then  (3.12) and (4.29) reduce
to $\omega_{uu}^1=c=0$. The second equations of (4.29)-(4.31) show
that
\begin{align}
&\frac{\omega_{pp}^1}{\lambda_p}=\frac{\omega_{qq}^1}{\lambda_q}=\frac{\omega_{rr}^1}{\lambda_r}=\varphi,\\
&\frac{\omega_{vv}^1}{\lambda_v}=-\frac{1}{\varphi}.
\end{align}
By taking into account (4.36) and (4.37) together with (3.11) and
(4.1), a very similar and direct computation as Case A also
gives a polynomial concerning function $\lambda_1$ with constant
coefficients. Hence, this is a contradiction and the mean curvature
$H$ has to be constant.

In conclusion, we complete the proof of Theorem 1.1.

\medskip

\noindent {\bf Acknowledgement:} {The authors would like to thank
Professor Cezar Oniciuc for fruitful discussions and useful
suggestions. The authors also wish to express appreciation to the anonymous referee for his helpful suggestions to improve the original version of this paper. The first
author was supported by NSFC (No.11601068), CSC of China (Grant No.201508210004)
 and China Postdoctoral Science Foundation (No.2016T90226, No.2014M560216).
The second author was supported by the Australian Research Council
grant (DP150101275).}


\end{document}